\documentclass{article} 
\usepackage{nips15submit_e,times}
\usepackage{bm}
\usepackage{hyperref}
\usepackage{url}
\usepackage{amsmath}
\usepackage{amsfonts}
\usepackage{amsopn}
\usepackage{amsthm}
\usepackage{graphicx}
\usepackage{subfigure}

\title{Probabilistic Solvers for Partial Differential Equations}

\author{
    Ilias Bilionis\thanks{\texttt{www.predictivesciencelab.org}} \\
School of Mechanical Engineering\\
Purdue University\\
585 Purdue Mall\\
West Lafayette, IN 47907\\
\texttt{ibilion@purdue.edu}
}

\newcommand{\bx}{\mathbf{x}}
\newcommand{\bC}{\mathbf{C}}
\newcommand{\bX}{\mathbf{X}}

\newcommand{\R}{\mathbb{R}}
\newcommand{\calL}{\mathcal{L}}
\newcommand{\calB}{\mathcal{B}}
\newcommand{\balpha}{\boldsymbol{\alpha}}
\newcommand{\bbeta}{\boldsymbol{\beta}}
\newcommand{\qref}[1]{Eq.~(\ref{eqn:#1})}
\newcommand{\sref}[1]{Sec.~\ref{sec:#1}}
\newcommand{\GP}{\operatorname{GP}}
\newcommand{\E}{\operatorname{E}}

\newcommand{\Cov}{\operatorname{Cov}}

\newtheorem{theorem}{Theorem}

\nipsfinalcopy 

\begin{document}

\maketitle

\begin{abstract}
	This work is concerned with the quantification of the epistemic uncertainties induced the discretization of partial differential equations.
	Following the paradigm of probabilistic numerics, we quantify this uncertainty probabilistically.
	Namely, we develop a  probabilistic solver suitable for linear partial differential equations (PDE) with mixed (Dirichlet and Neumann) boundary conditions defined on arbitrary geometries.
	The idea is to assign a probability measure on the space of solutions of the PDE and then condition this measure by enforcing that the PDE and the boundary conditions are satisfied at a finite set of spatial locations.
	The resulting posterior probability measure quantifies our state of knowledge about the solution of the problem given this finite discretization.
\end{abstract}

\section{Introduction}

We propose a novel discretization technique for partial differential equations (PDEs) that is inherently probabilistic and thus capable of quantifying the epistemic uncertainties induced by the discretization. 
Inspired by Diaconis \cite{Diaconis:1988ip}, we 1) Assign a prior probability measure on the space of solutions; 2) Observe the PDE on a discretized geometry; and 3) Use Bayes rule to derive the probability measure on the space of solutions that is compatible with both our prior beliefs and the observations. 
The uncertainty in this posterior probability measure is exactly the discretization uncertainty we are looking for. 

Our work fits within the program of probabilistic numerics (PN) \cite{Hennig:2015jf}, a term referring to the probabilistic interpretation of numerical algorithms.
The ideas of PN can be found in various fields, e.g., the quantification epistemic uncertainty in quadrature rules \cite{ohagan1991,kennedy1998,minka2000}, 
uncertainty propagation \cite{haylock1996,oakley2002}, 
sensitivity indexes \cite{becker2012,daneshkhah2013}, linear algebra \cite{hennig2015b}, ordinary differential equations \cite{skilling1992,graepel2003,calderhead2009,chkrebtii2013,barber2014,hennig2014,schober2014,conrad2015}, and optimization \cite{hennig2013}. 

The outline of the paper is as follows.
In \sref{related} we discuss how our work relates to existing methods for solving PDEs focusing mostly on PN-derived ones.
In \sref{metho} we present our solution methodology for arbitrary linear PDEs with mixed linear boundary conditions on arbitrary geometries.
In \sref{results} we employ our method to solve 1D and 2D boundary value problems with quantified discretization errors.
Finally, in \sref{conclusions} we present our conclusions.

\section{Related Work}
\label{sec:related}

There have been a few attempts  to quantify the numerical errors in PDEs, albeit they either ignore spatial discretization errors \cite{chkrebtii2013,conrad2015}, or they only work in square domains \cite{graepel2003}. 
The seminal work of \cite{Owhadi:2015vo} uses probabilistic arguments to derive a fast multigrid algorithm for the solution of elliptic PDEs, but it does not attempt to quantify epistemic uncertainties induced by discretization errors. 
Similarly to \cite{graepel2003}, the method we propose bears some resemblance to classical collocation methods \cite{ames1977}, but contrary to them it does not require the analytic specification of a set of basis functions.

\section{Methodology}
\label{sec:metho}

Consider the time-indpendent partial differential equation (PDE):
\begin{equation}
\label{eqn:pde}
\mathcal{L}[u](\bx) = f(\bx),\;\bx\in\Omega,
\end{equation}
with boundary conditions:
\begin{equation}
\label{eqn:bndry}
\mathcal{B}[u](\bx) = g(\bx),\;\bx\in\partial\Omega,
\end{equation}
where $\Omega\subset\R^d$ is a physical object in $d$ dimensions
($d=1,2$, or $3$), with a sufficiently smooth boundary
$\partial\Omega$, and $f$ and $g$ are the source and and the boundary
terms of the PDE.
$\calL$ and $\calB$ are linear differential operators acting on functions $u$,
with $\calB$ being at least one order less than $\calL$.
We say that \qref{pde} and \qref{bndry} constitute a boundary value problem
(BVP).

\paragraph{Multi-index notation for defining the differential operators.}
To make these concepts mathematically precise without having an overwhelming
notation, we will have to use the multi-index notation.
A multi-index is a $d$ tuple
\begin{equation}
\label{eqn:alpha}
\balpha = (\alpha_1,\dots,\alpha_d),
\end{equation}
of non-negative integers. We define the absolute value of a multi-index
$\balpha$ to be:
\begin{equation}
|\balpha| = \alpha_1 + \dots + \alpha_d,
\label{eqn:abs_alpha}
\end{equation}
and the $\balpha$-partial derivative operator:
\begin{equation}
\label{eqn:partial_alpha}
\partial^{\balpha} = \partial_1^{\alpha_i}\dots\partial_d^{\alpha_d},
\end{equation}
where
\begin{equation}
\label{eqn:partial_i_alpha_i}
\partial_r^{\alpha_r} = \frac{\partial^{\alpha_r}}{\partial x_r^{\alpha_r}},
\end{equation}
$\partial_i^0$ being the identity operator, $\partial_i^0u = u$.
Our assumption is that that $\calL$ is a $k$-th linear differential
operator, i.e.,
\begin{equation}
\calL[u] = \sum_{|\balpha|\le k}\calL_{\balpha}\partial^{\balpha} u,
\end{equation}
where the coefficients $\calL_{\balpha}$ are functions of $\bx$.
Similarly, $\calB$ is a $(k-1)$-th order linear differential operator,
\begin{equation}
\calB[u] = \sum_{|\balpha|\le k - 1}\calB_{\balpha}\partial^{\balpha}u,
\end{equation}
where the $\calB_{\balpha}$ are also functions of $\bx$.

\paragraph{Prior probability measure on the space of solutions.}
Prior to making any observations, we express our uncertainty about the 
solution $u$ of the BVP by assigning a Gaussian random field to it.
Namely, we assume that $u$ is a zero mean Gaussian process (GP) with
covariance function $c$, i.e.
\begin{equation}
\label{eqn:prior_u}
u\sim \GP\left(u \middle|0, c\right).
\end{equation}
Intuitively, we are assuming that, a priori, any sample from \qref{prior_u}
could be the solution to the BVP.
The selection of $c$ is a statement about the assumed smoothness of the
solution, its lengthscale of variation, etc.

We can now state and prove the following important result:
\begin{theorem}
	\label{thm:differential_u}
	Let $\calL'$ and $\calB'$ be identical to $\calL$ and $\calB$,
	respectively, but acting on $\bx'$ instead of $\bx$.
	If $\calL\calL'[c](\bx,\bx)$ and $\calB\calB'[c](\bx,\bx)$
	exist for all $(\bx,\bx)\in\R^{2d}$, then both $\calL[u](\bx)$  and
	$\calB[u](\bx)$ exist in the mean square sense for all $\bx\in\R^d$.
	Furthermore, the vector random process $(u, \calL[u], \calB[u])$ is Gaussian
	with zero mean and covariance:
	\begin{equation}
	\begin{array}{ccc}
	\Cov[u(\bx), u(\bx')] &=& c(\bx, \bx'),\\
	\Cov\left[\calL[u](\bx), u(\bx')\right] &=& \calL[c](\bx,\bx'),\\
	\Cov\left[\calB[u](\bx), u(\bx')\right] &=& \calB[c](\bx,\bx'),\\
	\Cov\left[\calL[u](\bx), \calL[u](\bx')\right] &=& \calL\calL'[c](\bx,\bx'),\\
	\Cov\left[\calL[u](\bx), \calB[u](\bx')\right] &=& \calL\calB'[c](\bx,\bx'),\\
	\Cov\left[\calB[u](\bx), \calB[u](\bx')\right] &=& \calB\calB'[c](\bx,\bx').
	\end{array}
	\end{equation}
	where $\calL'$ and $\calB'$ are identical to $\calL$ and $\calB'$,
	respectively, but they operate on functions of $\bx'$.
\end{theorem}
\begin{proof}
	This is a straightforward generalization of Theorem 2.2.2 of \cite{doi:10.1137/1.9780898718980}.
	It can be proved as follows. First, notice that any multi-indices $\balpha$ and
	$\bbeta$, the vector process $(u, \partial^{\balpha}, \partial^{\bbeta})$
	is Gaussian with mean zero and covariance:
	\begin{equation}
	\Cov[\partial^{\balpha}u(\bx), \partial^{\bbeta}u(\bx')] = \partial^{\balpha}\left(\partial'\right)^{\bbeta}c(\bx,\bx').
	\end{equation}
	This can be proved by letting the expectation over the probability measure
	of $u$ get inside the limit in the definition of the partial derivatives.
	Combining this result with the fact that linear combinations of
	Gaussian random variables are Gaussian, one can show that any finite
	dimensional probability density function of the vector process is Gaussian.
	The result follows from a simple application of Kolmogorov's extension theorem.
	
	The covariances between the various quantities are straightforward to
	derive.
	For example:
	\begin{eqnarray*}
		\Cov\left[\calL[u](\bx),\calB[u](\bx')\right] &=& \E\left[\calL[u](\bx)\calB[u](\bx')\right],\\
		&=& \E\left[\left(\sum_{|\balpha|\le k}\calL_{\balpha}(\bx)\partial^{\balpha}u(\bx)\right)\left(\sum_{|\bbeta|\le k-1}\calB_{\bbeta}(\bx')\left(\partial'\right)^{\bbeta}u(\bx')\right)\right]\\
		&=&
		\sum_{|\balpha|\le k}\sum_{|\bbeta|\le k-1}\calL_{\balpha}(\bx)\calB_{\bbeta}(\bx')
		\E\left[\partial^{\balpha}u(\bx)\left(\partial'\right)^{\bbeta}u(\bx')\right]\\
		&=& 
		\sum_{|\balpha|\le k}\sum_{|\bbeta|\le k-1}\calL_{\balpha}(\bx)\calB_{\bbeta}(\bx')
		\partial^{\balpha}\left(\partial'\right)^{\bbeta}c(\bx,\bx')\\
		&=& \calL\calB'[c](\bx,\bx').
	\end{eqnarray*}
\end{proof}

\paragraph{Constraining the prior measure using the BVP.}
Following Diaconis' recipe, we use Bayes rule to constrain 
\qref{prior_u} to functions that satisfy the BVP. 
Of course, it is impossible to condition on an uncountable infinity of points.
Instead, consider a set of $n^i$ points in the interior of $\Omega$,
\begin{equation}
\label{eqn:xi}
\bX^i = \left\{\bx_{1}^i,\dots,\bx_{n^i}^i\right\}\subset\Omega,
\end{equation}
and a set of $n^b$ points on the boundary $\partial \Omega$,
\begin{equation}
\label{eqn:bndry}
\bX^b = \left\{\bx_1^b,\dots,\bx_{n^b}^b\right\}\subset\partial\Omega.
\end{equation}
We wish to derive the posterior probability measure of $u$ conditioned on
\begin{equation}
\label{eqn:obs_L}
\calL[u]\left(\bx^i_j\right) = f\left(\bx^i_j\right),\;j=1,\dots,n^i,
\end{equation}
and
\begin{equation}
\label{eqn:obs_B}
\calB[u]\left(\bx^b_j\right) = g\left(\bx^b_j\right),\;
j=1,\dots,n^b.
\end{equation}

\begin{theorem}
	\label{thm:u_conditioned}
	The posterior of the GP $u$ conditioned on \qref{obs_L} and
	\qref{obs_B} is a GP with posterior mean:
	\begin{equation}
	\label{eqn:posterior_mean}
	\tilde{m}(\bx) = \mathbf{c}(\bx)^T\bC^{-1}\mathbf{y},
	\end{equation}
	and posterior covariance:
	\begin{equation}
	\label{eqn:posterior_covariance}
	\tilde{c}(\bx, \bx') = c(\bx, \bx') - \mathbf{c}(\bx)^T\bC^{-1}\mathbf{c}(\bx'),
	\end{equation}
	with
	\begin{equation}
	\bC = \left(
	\begin{array}{cc}
	\calL\calL'c\left(\bX^i,\bX^i\right) & \calL\calB'c\left(\bX^i, \bX^b\right)\\
	\left(\calL\calB'c(\bX^i, \bX^b)\right)^T & \calB\calB'c\left(\bX^b,\bX^b\right) 
	\end{array}
	\right),
	\end{equation}
	\begin{equation}
	\mathbf{c}(\bx) =
	\left(\begin{array}{c}
	\calL c(\bx, \bX^i)\\
	\calB c(\bx, \bX^b)
	\end{array}\right),
	\end{equation}
	\begin{equation}
	\mathbf{y} = \left(
	\begin{array}{c}
	f\left(\bX^i\right)\\
	g\left(\bX^b\right)
	\end{array}
	\right),
	\end{equation}
	where for any function $h:\R^d\rightarrow R$,
	and sets of $n$ and $n'$ points $\bX=\{\bx_1,\dots,\bx_{n}\}$ and $\bX'=\{\bx_1',\dots,\bx_{n'}'\}$, respectively, we define
	$\mathbf{h}(\bX, \bX')$ to be the matrix with elements $(h(\bx_j, \bx_r'))$.
\end{theorem}
\begin{proof}
	According to Theorem~\ref{thm:differential_u}, the vector random process
	$(u, \calL[u], \calB[u])$ is Gaussian.
	Thus, for any set of $n^t$ test points, $\bX^t=\{\bx_1^t,\dots,\bx_{n^t}^t\}$
	the joint distribution of the random vector $\left(u(\bX^t), \calL[u](\bX^i), \calB[u](\bX^b)\right)$
	is Gaussian.
	The result follows by conditioning this joint distribution on the 
	observed values of $\calL[u](\bX^i)$ and $\calB[u](\bX^b)$,
	and then making use of the definition of a GP.
\end{proof}

\paragraph{Probabilistic solution of BVPs and quantification of the discretization error.}
The conditioned GP of Theorem~\ref{thm:u_conditioned} captures our state of
knowledge after enforcing the PDE and the boundary condition on a finite 
discretization of the spatial domain.
We call the posterior mean, $\tilde{m}$ is as in \qref{posterior_mean},
the \emph{probabilistic solution} to the BVP.
The pointwise uncertainty of this solution that corresponds to the
discretization error is neatly captured by the pointwise posterior variance:
\begin{equation}
\label{eqn:posterior_variance}
\sigma^2(\bx) = \tilde{c}(\bx, \bx),
\end{equation}
where $\tilde{c}$ is as in \qref{posterior_covariance}.

\paragraph{Time-dependent PDEs.}
Generalization of our methodology to time-dependent PDEs is
straightfoward. All that one needs to do is augment $\bx$ with $t$, replace
$\Omega$ with $\Omega\times [0,+\infty]$, and make the initial condition 
part of the boundary condition.
However, this generalization introduces computational challenges that are
beyond the scope of this work.
Thus, we do not consider time-dependent PDEs in this paper.

\paragraph{Connections to the finite element method.}
A natural question that arises is ``What combination of covariance function and observed data gives rise to established numerical methods for the solution of PDEs such as finite difference, finite volume, and finite element methods \cite{ames1977}?"
Establishing such a connection would automatically provide  probabilistic errors bars for existing numerical codes.
Even though a complete survey of this topic is beyond the scope of the paper, we feel obliged to report an important theoretical observation that reveals a connection between the finite element method (FEM) and PN.
The observation that we have made is that the mean of the posterior GP becomes exactly the same as the FEM solution to the PDE if 1) the covariance function $c$ is picked to be the Green's function of the operator $\calL$~\cite{selcuk2006}; and 3) the observations upon which we condition the prior measure are the integrals $\int_{\Omega}\calL[u](\bx)\phi_i(\bx)d\bx=\int_{\Omega}f(\bx)\phi_i(\bx)d\bx, i=1,\dots,n_e$, where $\{\phi_i,\i=1,\dots,n_e)\}$ is the finite element basis upon which the solution is expanded.
Unfortunately, in this case the posterior variance of the GP turns out to be infinite and consequently, not very useful for quantifying the epistemic uncertainties due to the discretization.

\section{Numerical Results}
\label{sec:results}

We use the ideas discussed above to solve 1D and 2D boundary value problems.
We show how our ideas can naturally be applied to arbitrary geometries
without the use of a mesh.
Throughout this study we use a stationary covariance function for $u$:
\begin{equation}
\label{eqn:covar}
c(\bx, \bx') = s^2\exp\left\{-\frac{1}{2}\sum_{r=1}^d\frac{\left(x_r - x_r'\right)^2}{\ell_r^2}\right\},
\end{equation}
where $s$ and $\ell_r,r=1,\dots,d$ are parameters that can be interpreted as
the signal strength and the lengthscale of spatial dimension $r$, respectively.
The choice of this covariance function corresponds to a priori knowledge that
$u$ is infinitely differentiable.
The code has been implement in Python using Theano~\cite{2016arXiv160502688short} for the symbolic computation of all the required derivatives.

%
%

\subsubsection*{Case Study 1: 1D Steady State Heat Equation with Neumann Boundary Condition}
\label{sec:heat_1}

Consider the boundary value problem:
\begin{equation}
\label{eqn:heat_1}
-\frac{d}{dx}\left(a(x)\frac{du}{dx}\right) - \frac{1}{2}u = e^{-(x-2)^2} 
\end{equation}
with
\begin{equation}
\label{eqn:heat_1_cond}
a(x) = \frac{1}{2}\arctan\left(20(x - 1)\right) + 1,
\end{equation}
and boundary conditions:
\begin{equation}
\label{eqn:heat_1_bndry}
\frac{du(0)}{dx} = 0,\;\mbox{and}\;u(3) = 0.
\end{equation}
In this case, we assume that $s=2$ in \qref{covar}.
However, instead of assuming a value for $\ell$ we estimate by maximizing
the likelihood of the data \qref{obs_L} and \qref{obs_B}.
Despite the fact that this can be done with a fast algorithm, here we just
used an exhaustive search because of the simplicity of the problem.
In Figure~\ref{fig:heat_1}~(a), we compare the probabilistic solution 
using $n_i=20$ (green dashed-line), $n_i=40$ (red dotted line),
and $n_i=80$ (magenta dash-dotted line) equidistant internal 
discretization points to the exact solution (obtained by the finite difference method using $n=10,000$
discretization points). 
The shaded areas correspond to
95\% predictive intervals that capture the discretization error.
In Figure~\ref{fig:heat_1}~(b), we show the normalized likelihood
(likelihood divided by its maximum value) for each one of the cases considered.

\begin{figure}[htp]
	\centering
	\subfigure[]{
		\includegraphics[width=0.45\textwidth]{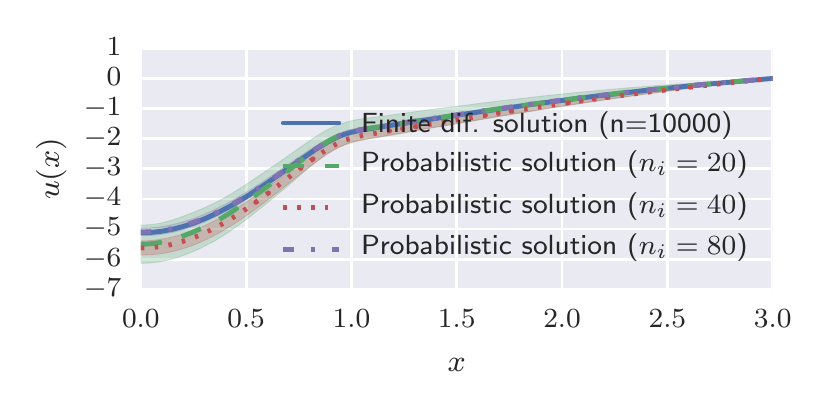}
	}
	\subfigure[]{
		\includegraphics[width=0.45\textwidth]{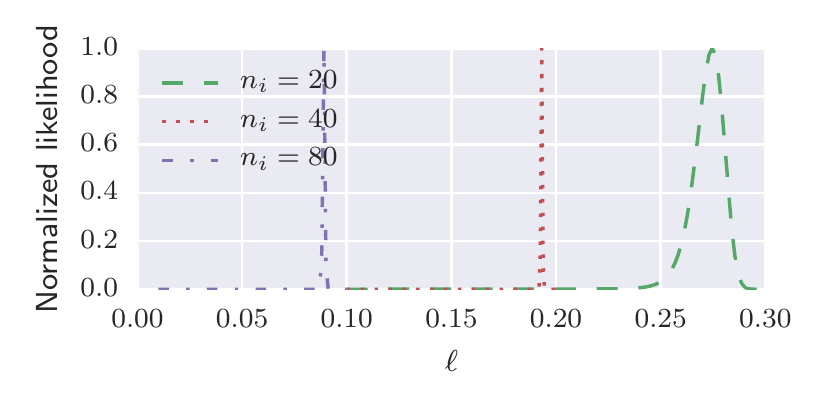}
	}
	\caption{Case study 2 (\qref{heat_1} subject to \qref{heat_1_bndry}).
		Subfigure~(a) compares three probabilistic solutions to the
		finite difference solution.
		Subfigure~(b) depicts the normalized likelihood as a function of
		the lengthscale $\ell$ of \qref{covar}.
	}
	\label{fig:heat_1}
\end{figure}

\subsubsection*{Case Study 2: Elliptic Equation on a Disk}
\label{sec:elliptic_disk}

Consider the elliptic partial differential equation on the unit disk,
\begin{equation}
\label{eqn:unit_disk}
\Omega = \left\{(x_1, x_2): x_1^2 + x_2^2 \le 1\right\},
\end{equation}
\begin{equation}
\label{eqn:elliptic_disk}
-\nabla^2u = 1,\;\mbox{for}\;(x_1,x_2)\in\Omega,
\end{equation}
with boundary conditions:
\begin{equation}
\label{eqn:elliptic_disk_bndry}
u(x_1, x_2) = 0,\;\mbox{for}\;(x_1,x_2)\in\partial\Omega.
\end{equation}
The exact solution is:
\begin{equation}
\label{eqn:elliptic_disk_exact}
u(x_1,x_2) = \frac{1 - x_1^2 - x_2^2}{4}.
\end{equation}
We use $n_i=16$ internal and $n_b=5$ boundary discretization points.
We assume that the $s=0.1$ and we choose $\ell$ so that the likelihood of
\qref{obs_L} and \qref{obs_B} is maximized ($\ell = 3.5$ for the chosen discretization).
Figure~\ref{fig:elliptic_disk}~(a) and~(b) depict contours of the exact and the
probabilistic solution, respectively.
Figure~\ref{fig:elliptic_disk}~(a) and~(b) depict contours of the exact and the
probabilistic solution, respectively.
In subfigure~(b), the internal points are indicated by blue crosses and the
boundary points by red disks.
Subfigure~(c) depicts the exact absolute value of the discretization error
while~(d) plots the contours of predicted discretization error (two times
the square root of the posterior variance \qref{posterior_variance}).

\begin{figure}[tbh]
	\centering
	\subfigure[Exact solution.]{
		\includegraphics[width=0.48\textwidth]{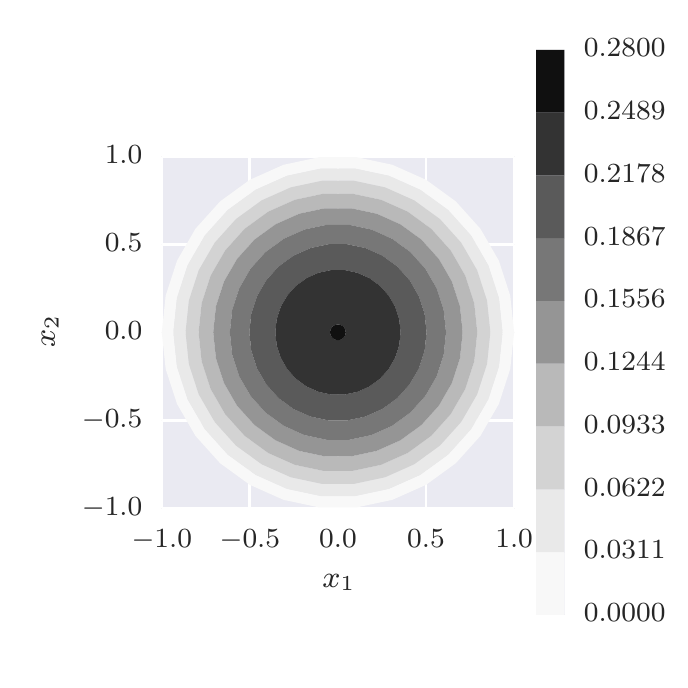}
	}
	\subfigure[Probabilsitic solution.]{
		\includegraphics[width=0.48\textwidth]{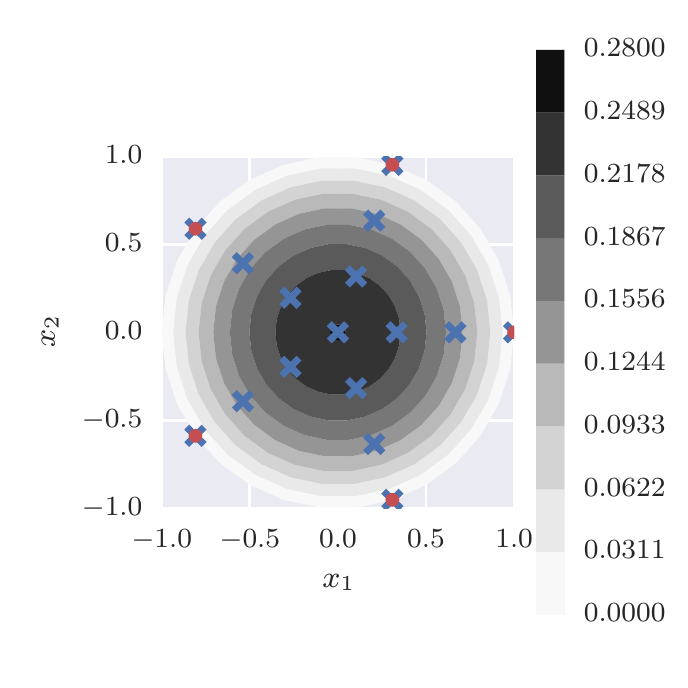}
	}
	\subfigure[Exact absolute error.]{
		\includegraphics[width=0.48\textwidth]{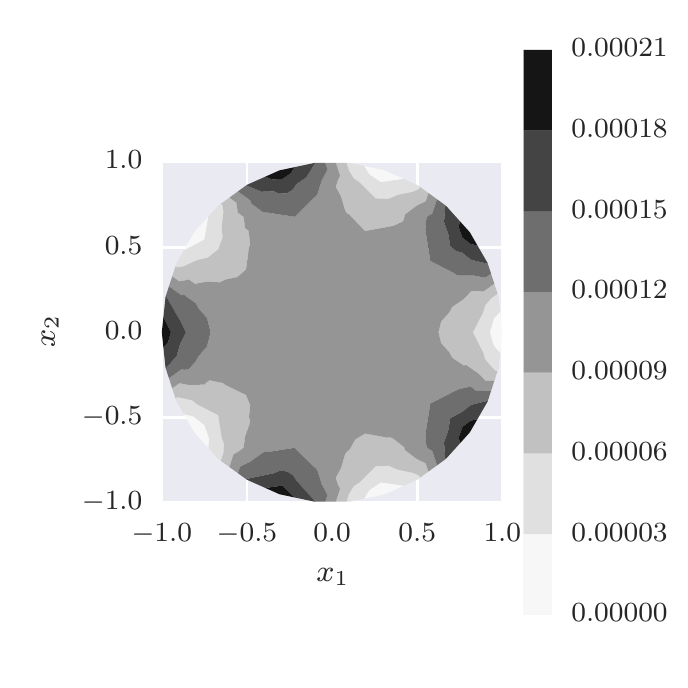}
	}
	\subfigure[Predicted discretization error.]{
		\includegraphics[width=0.48\textwidth]{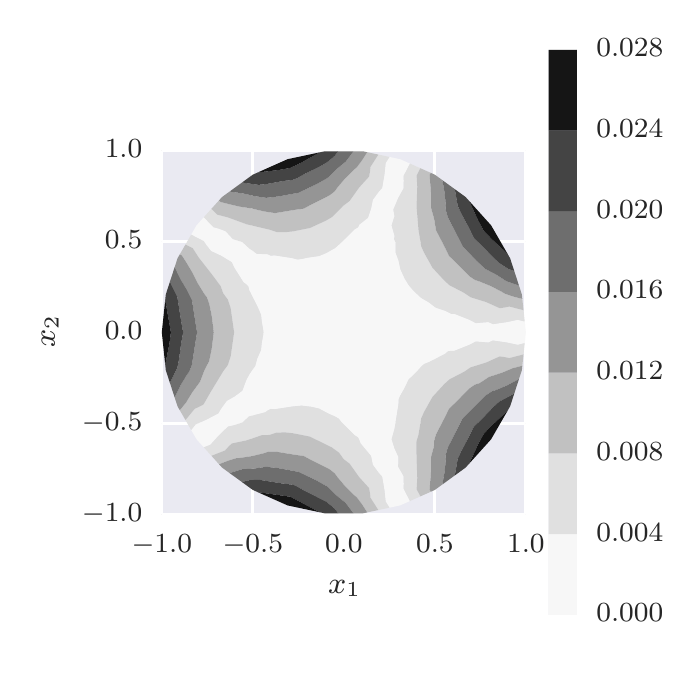}
	}
	\caption{Case study 4 (\qref{elliptic_disk} with \qref{elliptic_disk_bndry}).}
	\label{fig:elliptic_disk}
\end{figure}

\subsubsection*{Case Study 3: Elliptic Equation on a Disk with Unknown Solution}
\label{sec:elliptic_disk_unknown}

Consider the elliptic partial differential equation on the unit disk
($\Omega$ as in \qref{unit_disk}),
\begin{equation}
\label{eqn:elliptic_disk_unknown}
-\nabla^2u = 4\exp\left\{-\frac{1}{2}\left(\frac{Rx_1 - x_{01}}{\lambda}\right)^2
-\frac{1}{2}\left(\frac{Rx_2 - x_{02}}{\lambda}\right)^2
\right\},\;\mbox{for}\;(x_1,x_2)\in\Omega,
\end{equation}
with $R=0.3, \sigma=0.025, x_{01} = 0.6R\cos(0.2)$, and $x_{02}=0.6R\sin(0.2)$,
with boundary conditions:
\begin{equation}
\label{eqn:elliptic_disk_unknown_bndry}
u(x_1, x_2) = 0,\;\mbox{for}\;(x_1,x_2)\in\partial\Omega.
\end{equation}
We use $n_i=50$ internal and $n_b=20$ boundary discretization points.
We assume that the $s=0.01$ and we choose $\ell$ so that the likelihood of
\qref{obs_L} and \qref{obs_B} is maximized ($\ell = 0.26$ for the chosen discretization).
Figure~\ref{fig:elliptic_disk_unknown}~(a) and~(b) depict contours of the exact and the
probabilistic solution, respectively.
In subfigure~(b), the internal points are indicated by blue crosses and the
boundary points by red disks.

\begin{figure}[tbh]
	\centering
	\subfigure[Probabilistic solution.]{
		\includegraphics[width=0.48\textwidth]{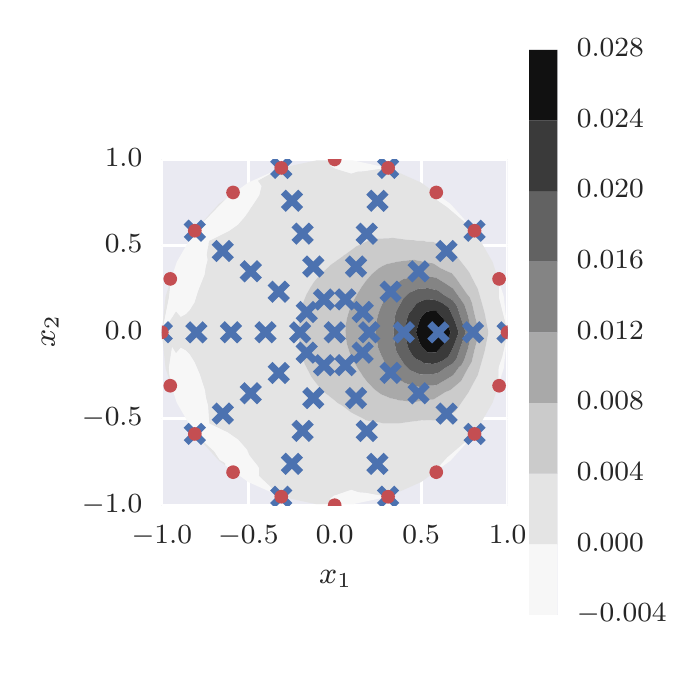}
	}
	\subfigure[Predicted discretization error.]{
		\includegraphics[width=0.48\textwidth]{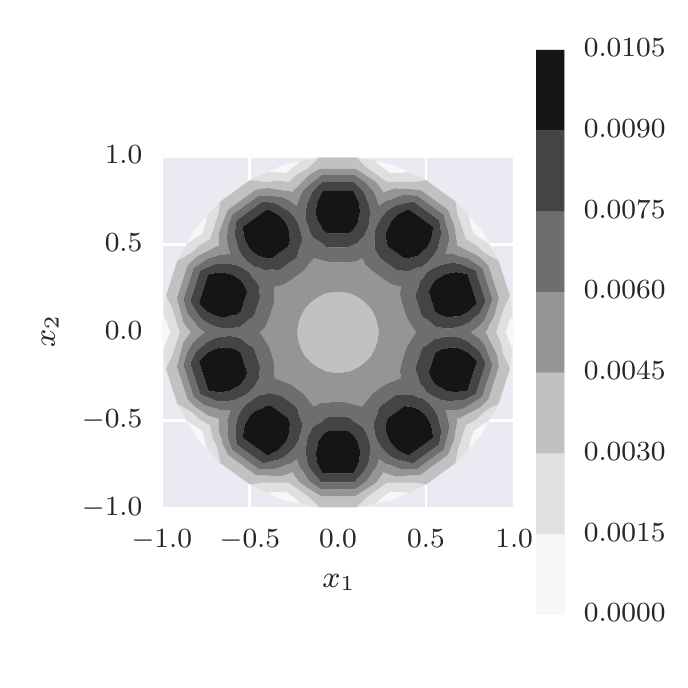}
	}
	\caption{Case study 5 (\qref{elliptic_disk_unknown} with \qref{elliptic_disk_unknown_bndry}).}
	\label{fig:elliptic_disk_unknown}
\end{figure}

\subsubsection*{Case Study 4: Elliptic Equation on Arbitrary Shape}
\label{sec:elliptic_arbitrary}

We consider an elliptic partial differential equation on
an arbitrary domain $\Omega$ (see Figure~\ref{fig:elliptic_arbitrary_shape}).
The equation is as in \qref{elliptic_disk_unknown}, but with
$R=0.8, \sigma=0.025, x_{01} = R\cos(\pi/4)$, and $x_{02}=R\sin(\pi/4)$.
We fix $n_b=20$ and $s=0.2$ and experiment with
$n_i=34, 43$, and $61$.
In each case we fit the $\ell$ by maximizing the likelihood of
\qref{obs_L} and \qref{obs_B}.

\begin{figure}[tbh]
	\centering
	\subfigure[Probabilistic solution ($n_i=34, \ell=0.18$).]{
		\includegraphics[width=0.48\textwidth]{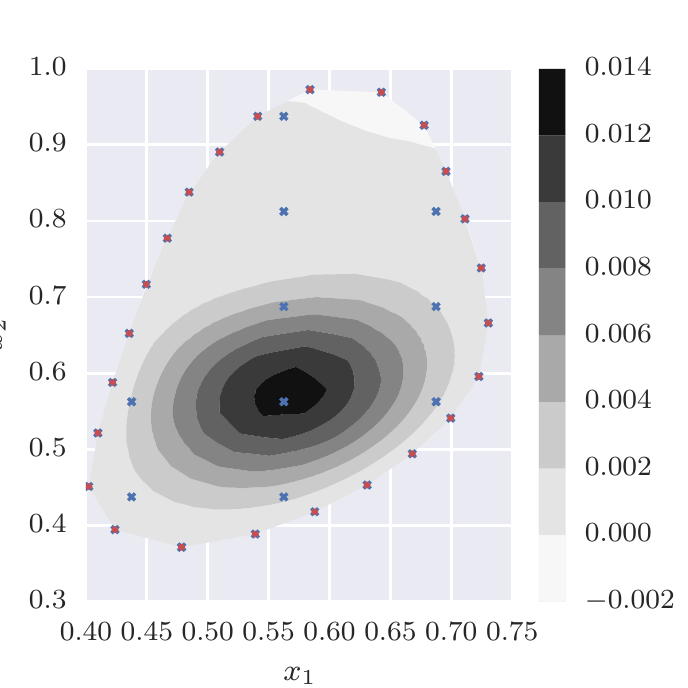}
	}
	\subfigure[Predicted discretization error ($n_i=34, \ell=0.18$).]{
		\includegraphics[width=0.48\textwidth]{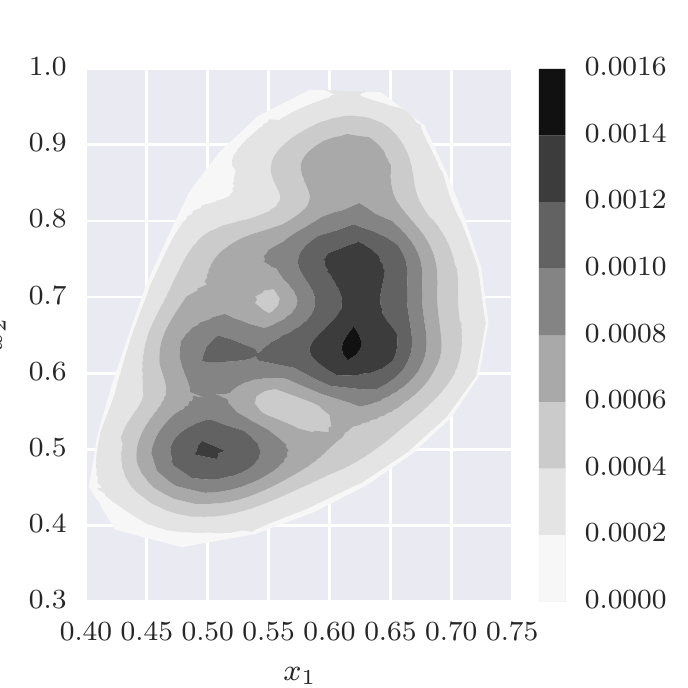}
	}
	\subfigure[Probabilistic solution ($n_i=43, \ell=0.14$).]{
		\includegraphics[width=0.48\textwidth]{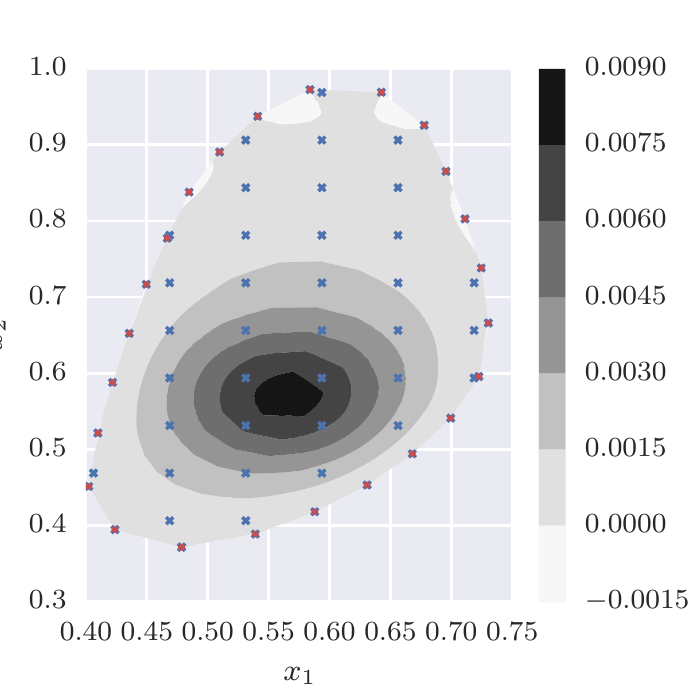}
	}
	\subfigure[Predicted discretization error ($n_i=43, \ell=0.14$).]{
		\includegraphics[width=0.48\textwidth]{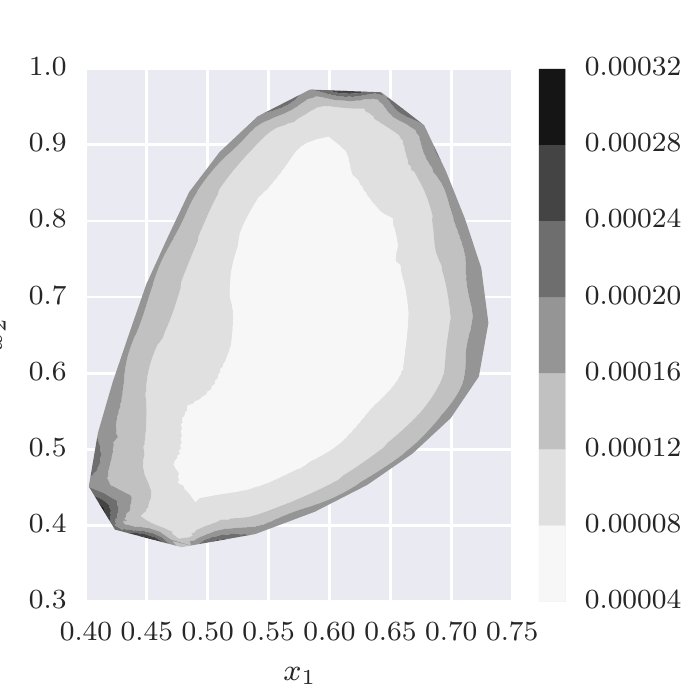}
	}
	\caption{Case study 4.}
	\label{fig:elliptic_arbitrary_shape}
\end{figure}

\section{Conclusions}
\label{sec:conclusions}
We presented a probabilistic method for solving linear PDEs with mixed linear boundary conditions on arbitrary geometries.
The idea is to assign a prior probability measure on the space of solutions, observe some discretized version of the PDE, and then derive the posterior probability measure.
Different choices of covariance functions and observed data result in different solution schemes.
The value of the method lies on the fact that it is capable of quantifying the epistemic uncertainties induced by the discretization scheme.
Such probabilistic schemes can find application in adaptive mesh refinement,  parallelized PDE solvers that are able to recover gracefully in case of sub-task failure, solution multi-scale/physics PDEs of which the right hand side depends on some expensive simulations that can only be performed a finite number of times, and more. 
The main drawback of the proposed methodology is that it requires the factorization of a dense covariance matrix.
This task can be prohibitively expensive when solving PDEs on realistic domains.
Further research is required to discover covariance functions that lead to sparse matrices that can be factorized efficiently.

\clearpage
\pagebreak

\bibliography{bibliography}
\bibliographystyle{unsrt}

\end{document}